\documentclass[a4paper,10pt,reqno, english]{amsart}  
\usepackage[utf8]{inputenc}
\usepackage[T1]{fontenc}
\usepackage{amsmath,amsthm}
\usepackage{amsfonts,amssymb,enumerate}
\usepackage{url,paralist}
\usepackage{mathtools}  
\usepackage[colorlinks=true,urlcolor=blue,linkcolor=red,citecolor=magenta]{hyperref}
\usepackage{enumerate}
\usepackage{anysize}

\theoremstyle{plain}
\newtheorem{theorem}{Theorem}[section]
\newtheorem{lemma}[theorem]{Lemma}

\newtheorem{corollary}[theorem]{Corollary}

\newtheorem{conjecture}[theorem]{Conjecture}

\newtheorem*{theorem*}{Theorem}

\newtheorem*{claim*}{Claim}

\theoremstyle{definition}

\newtheorem{example}[theorem]{Example}
\newtheorem{remark}[theorem]{Remark}

\newcommand{\F}{\mathcal{F}}
\newcommand{\R}{\mathbb{R}}
\newcommand{\Z}{\mathbb{Z}}
\newcommand{\pt}{\mathrm{pt}}

\newcommand{\conv}{\mathrm{conv}}
\newcommand{\KG}{\mathrm{KG}}

\renewcommand{\F}{\mathcal{F}}

\newcommand{\cd}{\mathrm{cd}}
\newcommand{\tcd}{\mathrm{tcd}}

\begin{document}

\title[Chromatic numbers of stable Kneser hypergraphs]{Chromatic numbers of stable Kneser hypergraphs via topological Tverberg-type theorems}



\author{Florian Frick}
\address{Mathematical Sciences Research Institute, 17 Gauss Way, Berkeley, CA 94720, USA}
\address{Dept.\ Math., Cornell University, Ithaca, NY 14853, USA}
\email{ff238@cornell.edu}

\date{\today}
\maketitle


\begin{abstract}
\small
Kneser's 1955 conjecture -- proven by Lov\'asz in 1978 -- asserts that in any partition of the
$k$-subsets of $\{1, 2, \dots, n\}$ into $n-2k-3$ parts, one part contains two disjoint sets. 
Schrijver showed that one can restrict to significantly fewer $k$-sets and still observe the 
same intersection pattern. Alon, Frankl, and Lov\'asz proved a different generalization of
Kneser's conjecture for $r$ pairwise disjoint sets. Dolnikov generalized Lov\'asz' result to arbitrary 
set systems, while K\v r\' i\v z did the same for the $r$-fold extension of Kneser's conjecture.
Here we prove a common generalization of all of these results. Moreover, we prove additional 
strengthenings by determining the chromatic number of certain sparse stable Kneser hypergraphs, 
and further develop a general approach to establishing lower bounds for chromatic numbers 
of hypergraphs using a combination of methods from equivariant topology and intersection 
results for convex hulls of points in Euclidean space. 
\end{abstract}

\section{Introduction}

Splitting $\binom{[n]}{k}$, the set of $k$-element subsets of $[n] = \{1, 2, \dots, n\}$, into $n-2k-3$ parts,
one of the parts must contain two disjoint $k$-sets. This was conjectured by Kneser~\cite{kneser1955} and proven by Lov\'asz~\cite{lovasz1978}.
This statement about the intersection pattern of all $k$-element subsets of $[n]$ easily translates into
a graph coloring problem: construct the \emph{Kneser graph} $\KG(k,n)$ on vertex set $\binom{[n]}{k}$
with an edge between any two vertices that correspond to disjoint $k$-sets. Kneser's conjecture then 
says that the chromatic number $\chi(\KG(k,n))$, that is, the least number of colors needed to color the
vertices such that any edge has endpoints of distinct colors, is at least~${n -2(k-1)}$. A greedy coloring 
shows that this lower bound is optimal. 

This has been generalized in several different ways. \emph{Firstly}, Schrijver~\cite{schrijver1978} showed that if one restricts the vertex set
of~$\KG(k,n)$ to \emph{stable} $k$-element sets, that is, those $k$-sets $\sigma \subset [n]$ that do not
contain two successive elements in cyclic order, the chromatic number does not decrease. We denote the
induced subgraph of stable $k$-sets by~$\KG(k,n)_{2-\textrm{stab}}$. \emph{Secondly}, Alon, Frankl, and Lov\'asz~\cite{alon1986} showed
that in any partition of $\binom{[n]}{k}$ into $\lceil \frac{n-r(k-1)}{r-1} \rceil-1$ parts, there is one part that
contains $r$ pairwise disjoint $k$-sets, proving a conjecture of Erd\H os~\cite{erdos1973}. In the same way that Kneser's
question translates into a graph coloring problem Alon, Frankl, and Lov\'asz's result establishes a lower
bound for the chromatic number of an associated hypergraph: For an integer $r \ge 2$ denote by $\KG^r(k,n)$ 
the \emph{Kneser hypergraph} with vertex set $\binom{[n]}{k}$ and a hyperedge of cardinality $r$ for
any $r$ sets $\sigma_1, \dots, \sigma_r \subset [n]$ of cardinality~$k$ that are pairwise disjoint. The 
chromatic number $\chi(H)$ of a hypergraph $H$ is the least number of colors needed to color its vertices 
such that no hyperedge is monochromatic, i.e., receives only one color. In this language Alon, Frankl, and Lov\'asz showed
that $\chi(\KG^r(k,n)) = \lceil \frac{n-r(k-1)}{r-1} \rceil$. \emph{Thirdly}, Dolnikov~\cite{dolnikov1988} gave general lower bounds for 
the chromatic number of graphs that arise from arbitrary set systems with vertices corresponding to sets and 
edges to disjoint sets. His bound specializes to Kneser's conjecture for the system of $k$-subsets of~$[n]$.
K\v r\' i\v z~\cite{kriz1992} extended this to the hypergraph setting, establishing bounds for chromatic numbers that specialize to
the result of Alon, Frankl, and Lov\'asz. For a system $\F$ of subsets of~$[n]$ K\v r\' i\v z defines a combinatorial
quantity, the \emph{$r$-colorability defect} $\cd^r(\F)$, and shows
that in general $(r-1)\chi(\KG^r(\F)) \ge \cd^r(\F)$; see Section~\ref{sec:nonlocal} for definitions.

Since then some effort has been invested into combining the first two generalizations into one common
generalization of Kneser's conjecture. Meunier~\cite{meunier2011} showed that the chromatic number of $\KG^r(k,n)$ does
not decrease if one restricts to those vertices that correspond to \emph{almost stable} $k$-sets, where a
set $\sigma \subset [n]$ is almost stable if it does not contain two successive elements of~$[n]$ but may 
contain $1$ and~$n$. Alishahi and Hajiabolhassan~\cite{alishahi2015} showed that $\KG^r(k,n)_{2-\textrm{stab}}$, the induced 
subhypergraph of $\KG^r(k,n)$ that only contains those vertices corresponding to stable $k$-sets, 
still satisfies $\chi(\KG^r(k,n)_{2-\textrm{stab}}) =  \lceil \frac{n-r(k-1)}{r-1} \rceil$ unless $r$ is odd and $n \equiv k \mod r-1$.

Here we prove the common generalization of the results of Schrijver and of Alon, Frankl, and Lov\'asz
in full generality; see Theorem~\ref{thm:stable}:

\begin{theorem*}
	For any integers $r \ge 2$, $k\ge 1$, $n \ge rk$ we have that $$\chi(\KG^r(k,n)_{2-\textrm{stab}}) = \Big\lceil \frac{n-r(k-1)}{r-1} \Big\rceil.$$
\end{theorem*}

Our methods are different from those of~\cite{meunier2011} and~\cite{alishahi2015}.
We use a combination of methods from equivariant topology, in particular, geometric transversality results, and the geometry
of point sets in Euclidean space. This was already partially developed in the recent paper~\cite{frick2017} of the author; see also 
Sarkaria~\cite{sarkaria1990, sarkaria1991} for earlier related work.

In fact, it is possible to generalize these lower bounds from the family of $k$-subsets of~$[n]$ to arbitrary systems~$\F$
of subsets of~$[n]$, establishing K\v r\' i\v z' lower bound for the set system $\F_{2-\textrm{stab}}$, those sets in $\F$ that are stable;
see Theorem~\ref{thm:hyp-schrijver}:

\begin{theorem*}
	Let $\F$ be a system of subsets of~$[n]$ and $r \ge 3$. Then $$\chi(\KG^r(\F_{2-\textrm{stab}})) \ge \Big\lceil \frac{1}{r-1}\cd^r(\F) \Big\rceil.$$
\end{theorem*}

For $r = 2$ this lower bound only holds for the almost stable sets in~$\F$; see Theorem~\ref{thm:gen-schrijver}.
With our methods that relate chromatic numbers of hypergraphs to intersections of convex hulls of point sets in
Euclidean space, K\v r\' i\v z' result follows from putting points in strong general position, while for the theorem above we
place points along the moment curve.

While $\cd^r(\F)$ is a purely combinatorial quantity, we define the \emph{topological $r$-colorability defect} $\tcd^r(\F)$
that takes the topology of the set system $\F$ into account and satisfies $\tcd^r(\F) \ge \cd^r(\F)$, where the gap can be
arbitrarily large. We show the following improvement of K\v r\' i\v z' lower bound; see Theorem~\ref{thm:top}:

\begin{theorem*} 
	For $r$ a power of a prime
	$$\chi(\KG^r(\F)) \ge \Big\lceil \frac{1}{r-1}\tcd^r(\F) \Big\rceil.$$
\end{theorem*}

Interestingly, while this result easily implies K\v r\' i\v z' result for arbitrary $r$, the improved lower bound only holds for prime powers~$r$.
We further extend the theorem above to proper subhypergraphs of~$\KG^r(\F)$; see Theorems~\ref{thm:top-opt} and~\ref{thm:top-stable},
where the former is a common generalization of the three extensions of Kneser's conjecture mentioned above, which furthermore
detects global topological structure of the set system $\F$ in addition to the local combinatorial structure.

It has been conjectured by Ziegler~\cite{ziegler2002} and Alon, Drewnowski, and \L uczak~\cite{alon2009} that the bound
$\chi(\KG^r(k,n)) = \lceil \frac{n-r(k-1)}{r-1} \rceil$ still holds when the hypergraph $\KG^r(k,n)$ is restricted to vertices
corresponding to $r$-stable sets. A set $\sigma \subset [n]$ is \emph{$s$-stable} if any two elements of $\sigma$ are at distance
$\ge s$ in the cyclic order of~$[n]$. Schrijver's result is the $r=2$ case of this conjecture, and Alon, Drewnowski, and \L uczak
furthermore showed that if the conjecture holds for $r = p$ and $r = q$ then it also holds for their product $r =pq$. Thus the
conjectured lower bound holds for $r$ a power of two. Aside from this combinatorial reduction, no results for $s$-stable Kneser
hypergraphs for $s > 2$ were known. Here we prove optimal lower bounds for chromatic numbers of $s$-stable Kneser hypergraphs
$\chi(\KG^r(k,n)_{s-\textrm{stab}})$ for any~$s \ge 2$; see Corollary~\ref{cor:s-stable}:

\begin{theorem*}
	Let $s \ge 2$ be an integer and $r > 6s-6$ a prime power. Then for integers $k \ge 1$ and $n \ge rk$ we have that
	$$\chi(\KG^r(k,n)_{s-\textrm{stab}}) = \Big\lceil \frac{n-r(k-1)}{r-1} \Big\rceil.$$
\end{theorem*}

We prove extensions of this bound to arbitrary set systems; see Theorem~\ref{thm:top-stable}.
Section~\ref{sec:box} is devoted to proving further extensions of K\v r\' i\v z' bound to arbitrary $r$-uniform hypergraphs, 
extending the results for generalized Kneser hypergraphs. There we briefly relate our methods to the standard ``box complex approach.''

\section{The chromatic number of Kneser hypergraphs --- a simple proof}
\label{sec:simple}

As a toy example to showcase our method of relating chromatic numbers of hypergraphs to the geometry of point sets
in Euclidean space and geometric transversality results, we give a new and simple proof that $\chi(\KG^r(k,n)) 
= \lceil \frac{n-r(k-1)}{r-1} \rceil$. In~\cite{frick2017} the author gave a simple proof of the $r=2$ case, that is, Kneser's original conjecture,
by reducing it to the topological Radon theorem of Bajmoczy and B\'ar\'any~\cite{bajmoczy1979} 
that for any continuous map $f\colon \Delta_{d+1} \longrightarrow \R^d$
there are two points in disjoint faces of the $(d+1)$-simplex $\Delta_{d+1}$ that are identified by~$f$. Generalizations to the hypergraph setting, $r > 2$, 
use an $r$-fold analogue of the topological Radon theorem, the topological Tverberg theorem: For any continuous map 
$f\colon \Delta_{(r-1)(d+1)} \longrightarrow \R^d$ there are $r$ pairwise disjoint faces $\sigma_1, \dots, \sigma_r$ of~$\Delta_{(r-1)(d+1)}$
such that $f(\sigma_1) \cap \dots \cap f(\sigma_r) \ne \emptyset$ provided that $r$ is a power of a prime. This theorem is due to 
B\'ar\'any, Shlosman, and Sz\H ucz~\cite{barany1981} for prime~$r$, and due to \"Ozaydin~\cite{ozaydin1987} 
for the general case of prime powers. The same result fails for $r$ with at least two distinct prime divisors~\cite{blagojevic2015, frick2015, mabillard2015}.
In the following we will use that a generic map $f\colon \Delta_{(r-1)(d+1)-1} \longrightarrow \R^d$ does not identify points from
$r$ pairwise disjoint faces by a simple codimension count. If $f$ is an affine map then the correct notion of genericity for the vertex set
is strong general position; see Perles and Sigron~\cite{perles2014}.

Let $K$ be a simplicial complex on vertex set~$[n]$. A set $\sigma \subset [n]$ is called \emph{missing face} of~$K$ if
$\sigma \notin K$ and every proper subset of $\sigma$ is a face of~$K$. For an integer $r \ge 2$ the \emph{Kneser
hypergraph}~$\KG^r(K)$ of $K$ has as its vertex set the missing faces of~$K$, and a hyperedge for any $r$ pairwise disjoint missing
faces. More generally, for a simplicial complex $L \supset K$ the subhypergraph of $\KG^r(K)$ induced by those vertices that
correspond to missing faces of~$K$ that are faces of~$L$ is denoted by~$\KG^r(K, L)$.
We will need the following result, which is a simple consequence of the topological Tverberg theorem:

\begin{theorem}[{\cite[Cor.~4.7]{frick2017}}]
\label{thm:lower-bounds}
	Let $d \ge 0$ be an integer, and let $r\ge 2$ be a prime power. Let $K$ be a simplicial complex on at most $N=(r-1)(d+1)+1$ vertices such that 
	there exists a continuous map $f\colon K \longrightarrow \R^d$ with the property that for any $r$ pairwise 
	disjoint faces $\sigma_1, \dots, \sigma_r$ of $K$ we have that $f(\sigma_1) \cap \dots \cap f(\sigma_r) = \emptyset$.
	Then $\chi(\KG^r(K)) \ge \lfloor\frac{N-1}{r-1}\rfloor-d$.
\end{theorem}

As a simple consequence we obtain:

\begin{theorem}[Alon, Frankl, Lov\'asz~\cite{alon1986}]
\label{thm:afl}
	For any integers $r \ge 2$, $k\ge 1$, $n \ge rk$ we have that $$\chi(\KG^r(k,n)) = \Big\lceil \frac{n-r(k-1)}{r-1} \Big\rceil.$$
\end{theorem}

\begin{proof}
	Let $r \ge 2$ be a prime. Choose the integer $d$ such that $r(k - 1) - 1 \ge (r - 1)d > r(k - 2)$ and let
	$t = (r-1)d-r(k-2)-1$. Let $K = \Delta_{n-1}^{(k-2)} * \Delta_{t-1}$, that is, $K$ is the simplicial complex 
	on vertex set $[n+t]$ whose missing faces are precisely the $k$-element subsets of~$[n]$. Let 
	$f\colon K \longrightarrow \R^d$ be a generic map.
	
	Let $\sigma_1, \dots, \sigma_r$ be pairwise disjoint faces of~$K$. 
	They involve at most $r(k-1)+t = (r-1)(d+1)$ vertices.
	Thus the intersection $f(\sigma_1) \cap \dots \cap f(\sigma_r)$ is empty. Now the hypergraph $\KG^r(K)$ of missing faces
	of $K$ is~$\KG^r(k,n)$, and thus $\chi(\KG^r(k,n)) \ge \lfloor \frac{n+t-1}{r-1} \rfloor -d$. Write 
	$\lfloor \frac{n+t-1}{r-1} \rfloor = \frac{n+t-1 - \alpha}{r-1}$ for some $0 \le \alpha \le r-2$, and furthermore we have that
	$d = \frac{r(k-2)+t+1}{r-1}$. Thus
	$$\chi(\KG^r(k,n)) \ge  \frac{n+t-1 - \alpha}{r-1} - \frac{r(k-2)+t+1}{r-1} = \frac{n-r(k-2)-(\alpha+2)}{r-1}
	\ge \frac{n-r(k-1)}{r-1}.$$
	
	The case for general $r$ follows by a simple induction on the number of prime divisors as in~\cite{alon1986}.
\end{proof}

The proof above provides us with two ways to fine-tune in order to obtain stronger results. 
First, the simplicial complex $K$ does not have a point of $r$-fold coincidence among its pairwise disjoint faces
for codimension reasons. Adding a face to $K$ while preserving this property either deletes a missing face from $K$
and thus a vertex from $\KG^r(K)$ while not decreasing its chromatic number, or increases the size of a missing face
by one, which potentially deletes hyperedges from~$\KG^r(K)$ while preserving~$\chi(\KG^r(K))$.

Secondly, the lower bound for the chromatic number builds on a contradiction to the topological Tverberg theorem.
Proper strengthenings are due to Vu\v ci\' c and \v Zivaljevi\'c~\cite{vucic1993}, Hell~\cite{hell2008}, Engstr\"om~\cite{engstrom2011}, and 
Blagojevi\'c, Matschke, and Ziegler~\cite{blagojevic2009}. 
In Sections~\ref{sec:geom-transversality} and~\ref{sec:nonlocal} we will leverage some of these results to obtain sparse subhypergraphs
of~$\KG^r(k,n)$ that still have the same chromatic number, and extend this approach to general Kneser hypergraphs.

\section{Lower bounds for chromatic numbers via geometric transversality}
\label{sec:geom-transversality}

In this section we will extend the reasoning of the previous section to prove tight lower bounds for $2$-stable and,
for $r \ge 5$, $r \notin \{6,12\}$, and for $r=4$ for certain parameters, for $3$-stable Kneser hypergraphs. The same methods also imply 
an approximation to Ziegler's conjecture that $\chi(\KG^r(k,n)_{r-\textrm{stab}}) = \chi(\KG^r(k,n))$ for~${n \ge rk}$. The proofs in this section all
follow the same general pattern that we already outlined for the proof of Theorem~\ref{thm:afl} and contain
ad-hoc constructions. We will present a more generalized framework that works for general Kneser hypergraphs 
and implies tight lower bounds for the chromatic number of $s$-stable Kneser hypergraphs, where $s$ grows linearly
with~$r$, in Section~\ref{sec:nonlocal}.
In order to find sparse subhypergraphs of $\KG^r(k,n)$ that still require at least $\lceil \frac{n-r(k-1)}{r-1} \rceil$
colors in any proper coloring, we first need to recall:

\begin{theorem}[{\cite[Theorem~4.5]{frick2017}}]
\label{thm:transversality}
	Let $d, c \ge 0$ and $r \ge 2$ be integers, $K \subseteq L$ simplicial complexes such that
	for every continuous map $F\colon L \longrightarrow \R^{d+c}$ there are $r$ pairwise disjoint
	faces $\sigma_1, \dots, \sigma_r$ of $L$ such that $F(\sigma_1) \cap \dots \cap F(\sigma_r) \ne \emptyset$.
	Suppose $\chi(\KG^r(K, L)) \le c$. Then for every continuous map
	$f\colon K \longrightarrow \R^d$ there are $r$ pairwise disjoint faces $\sigma_1, \dots, \sigma_r$
	of $K$ such that $f(\sigma_1) \cap \dots \cap f(\sigma_r) \ne \emptyset$.
\end{theorem}

\noindent
We now need:
\begin{compactenum}[1.]
	\item \label{item1}
	A large simplicial complex $K$ containing the $(k-2)$-skeleton of~$\Delta_{n-1}$ such that there exists a
	continuous map $f\colon K \longrightarrow \R^d$ that does not identify points in $r$ pairwise disjoint faces,
	where $d = \lceil \frac{r(k-2)+1}{r-1} \rceil$.
	\item \label{item2}
	A small simplicial complex $L \supset K$ such that any continuous map $f\colon L \longrightarrow \R^{D-1}$ identifies
	points from $r$ pairwise disjoint faces, where $D = \lceil \frac{n-1}{r-1} \rceil$.
\end{compactenum}

Concerning point~\ref{item1} good candidates are (subcomplexes of) cyclic polytopes of dimension~${2k-2}$. These polytopes 
arise as convex hulls of finitely many points on the \emph{moment curve} $\gamma(t) = (t, t^2, \dots, t^d)$. In fact, we 
will use that the intersection combinatorics of convex hulls of point sets on the stretched moment curve are understood,
provided that the points are sufficiently far apart; see Bukh, Loh, and Nivasch~\cite{bukh2016} (according to the authors 
of~\cite{bukh2016} this was independently observed by B\'ar\'any and P\'or as well as Mabillard and Wagner).
A partition of $\{1,2, \dots, (r-1)(d+1)+1\}$ into $r$ parts $X_1, \dots, X_r$ is called \emph{colorful} if for each $1 \le k \le d+1$ 
the set $Y_k = \{(r-1)(k-1)+1, \dots, (r-1)k+1\}$ satisfies $|Y_k \cap X_i| = 1$ for all~$i$. We say that a partition $X_1 \sqcup \dots \sqcup X_r$ 
of $\{1,2, \dots, (r-1)(d+1)+1\}$ \emph{occurs} as a Tverberg partition in a sequence $x_1, \dots, x_N$ of points in~$\R^d$ 
if there is a subsequence $x_{i_1}, \dots, x_{i_n}$ of length $n = (r-1)(d+1)+1$ such that the sets $\conv\{x_{i_k} \: | \: k \in X_j\}$ 
all share a common point. We now have the following lemma:

\begin{lemma}[Bukh, Loh, and Nivasch~\cite{bukh2016}]
\label{lem:colorful}
	There are arbitrarily long sequences of points in $\R^d$ such that the Tverberg partitions that occur are precisely the colorful ones. 
\end{lemma}

These point sets are spread along the stretched moment curve. 
Let $P \subset \R^d$ be a (sufficiently large) point set provided by the lemma above. Let $X_1, \dots, X_r \subset P$ be pairwise disjoint sets
with $|\bigcup_i X_i| \le (r-1)(d+1)$. Then the intersection $\conv X_1 \cap \dots \cap \conv X_r$ is necessarily empty, since
otherwise we could find a point $p$ of $P \setminus \bigcup_i X_i$ and an index $j \in [r]$ such that adding $p$ to $X_j$ would
not be a colorful partition, but $\conv X_1 \cap \dots \cap \conv X_r \ne \emptyset$. Finding point sets with this particular property
is much simpler than Lemma~\ref{lem:colorful}: we need point sets in strong general position, which is a generic property; see 
Perles and Sigron~\cite{perles2014}.

For point~\ref{item2} above we will use the following proper extensions of the topological Tverberg theorem:

\begin{theorem}[Hell~\cite{hell2008}]
\label{thm:hell}
	Let $r \ge 5$ be a prime power, $d \ge 1$ an integer, and $N =(r-1)(d+1)$. Let $G$ be a vertex-disjoint union of cycles 
	and paths on the vertex set of~$\Delta_N$. Then for any continuous map $f \colon \Delta_N \longrightarrow \R^d$ there are $r$
	pairwise disjoint faces $\sigma_1, \dots, \sigma_r$ of~$\Delta_N$ such that $f(\sigma_1) \cap \dots \cap f(\sigma_r) \ne \emptyset$
	and no $\sigma_i$ contains both endpoints of an edge of~$G$. The same result holds for $r=4$ if $G$ is a disjoint
	union of paths or even for $r =3$ if $G$ is a disjoint union of edges.
\end{theorem}

\begin{theorem}[Blagojevi\'c, Matschke, and Ziegler~\cite{blagojevic2009}]
\label{thm:bmz}
	Let $r \ge 2$ be a prime, $d \ge 1$ an integer, and $N =(r-1)(d+1)$. Let $C_1 \sqcup \dots \sqcup C_m$ be a partition of the 
	vertex set of~$\Delta_N$ such that $|C_i| \le r-1$. Then for any continuous map $f \colon \Delta_N \longrightarrow \R^d$ there are $r$
	pairwise disjoint faces $\sigma_1, \dots, \sigma_r$ of~$\Delta_N$ such that $f(\sigma_1) \cap \dots \cap f(\sigma_r) \ne \emptyset$.
\end{theorem}

Finally, since these results come with restrictions on the number of prime divisors of~$r$, we need the following 
combinatorial reduction result:

\begin{lemma}[Alishahi and Hajiabolhassan~\cite{alishahi2015}]
\label{lem:primes}
	Let $r$, $s$ and $p$ be positive integers where $r \ge s \ge 2$ and $p$ is a prime number. 
	Assume that for any $n \ge rk$, $\chi(\KG^r(k,n)_{s-\textrm{stab}}) = \lceil \frac{n-r(k-1)}{r-1} \rceil$. 
	Then for any $n \ge prk$, we have $\chi(\KG^{pr}(k,n)_{s-\textrm{stab}}) = \lceil \frac{n-pr(k-1)}{pr-1} \rceil$.
\end{lemma}

This allows us to prove:

\begin{theorem}
\label{thm:stable}
	For any integers $r \ge 2$, $k\ge 1$, $n \ge rk$ we have that $$\chi(\KG^r(k,n)_{2-\textrm{stab}}) = \Big\lceil \frac{n-r(k-1)}{r-1} \Big\rceil.$$
\end{theorem} 

\begin{proof}
	According to Lemma~\ref{lem:primes} we can from now on assume that $r$ is a prime. We can also assume that $r \ge 3$,
	since the $r=2$ case is due to Schrijver~\cite{schrijver1978}. (In fact, the proof at hand works with minor changes for $r=2$;
	see~\cite{frick2017}.)
	Let $K$ be the simplicial complex on vertex set~$[n]$ defined by: $\sigma \subset [n]$ is a missing face of~$K$
	if and only if $|\sigma| = k$ and if $i \in \sigma$ for some $i \in [n-1]$ then $i+1\notin \sigma$, that is, the almost
	$2$-stable $k$-sets are the missing faces of~$K$.
	Choose the integer $d$ such that 
	$r(k - 1) - 1 \ge (r - 1)d > r(k - 2)$, that is, $d$ is the largest integer that satisfies $d \le \frac{r(k-1)-1}{r-1}$. Let
	$t = (r-1)d-r(k-2)-1$. Let $p_1, \dots, p_{n+t} \in \R^d$ be a sequence of points provided by Lemma~\ref{lem:colorful}.
	Now consider the complex $\overline K = K * \Delta_{t-1}$ for $t > 0$ or $\overline K = K$ if~${t = 0}$. Map the vertices 
	$[n]$ of $K$ bijectively to $p_1, \dots, p_n$ in the canonical order. Map the $t$ vertices of $\Delta_{t-1}$ to~$p_{n+1},
	\dots, p_{n+t}$ in arbitrary order. Extending linearly onto faces yields an affine map $f\colon \overline K \longrightarrow \R^d$.
	In fact, let us from here on not distinguish between the vertices of $\overline K$ and their images $p_1, \dots, p_{n+t}$.
	
	Let $\sigma_1, \dots, \sigma_r$ be pairwise disjoint faces of~$\overline K$ and suppose that $f(\sigma_1) \cap
	\dots \cap f(\sigma_r) \ne \emptyset$. W.l.o.g. the faces $\sigma_1, \dots, \sigma_r$ involve precisely $(r-1)(d+1)+1$ vertices
	by a codimension count. We now argue that for one of the faces $\sigma_i$ the intersection
	$\sigma_i \cap \{p_1, \dots, p_n\}$ has cardinality at least~$k$. 
	Indeed, if all intersections $\sigma_i \cap \{p_1, \dots, p_n\}$ had cardinality at most~$k-1$, then the 
	$\sigma_i$ would involve at most $r(k-1)+t = (r-1)(d+1)$ points, which is a contradiction to $f(\sigma_1) \cap
	\dots \cap f(\sigma_r) \ne \emptyset$.
	Thus one of the faces $\sigma_j$ satisfies $\sigma_j \cap \{p_1, \dots, p_n\} \ge k$. 
	If the intersection $f(\sigma_1) \cap \dots \cap f(\sigma_r)$ were nonempty, then by Lemma~\ref{lem:colorful} the
	partition $\sigma_1 \sqcup \dots \sqcup\sigma_r$ of the $(r-1)(d+1)+1$ points $\bigcup_i \sigma_i$ would be
	colorful. In particular, no two vertices of $\sigma_i$ can be successive in the linear order of~$[n+t]$. But then
	$\sigma_j$ cannot be a face of~$\overline K$, which is a contradiction. Hence the intersection
	$f(\sigma_1) \cap \dots \cap f(\sigma_r)$ is empty.
	
	Denote the vertex set of $\Delta_{t-1}$ by $v_1, \dots, v_t$. Let $L$ be the simplicial complex obtained by
	deleting the edge $(1,n)$ from the simplex on vertex set~${[n] \cup \{v_1, \dots, v_t\}}$. Let 
	$D = \lfloor \frac{n+t-1}{r-1} \rfloor$. Since $(r-1)D+1 \le n+t$, we have by Theorem~\ref{thm:hell} that for any continuous map 
	$F\colon L \longrightarrow \R^{D-1}$ there are $r$ points in $r$ pairwise disjoint faces of~$L$ whose images coincide. 
	By Theorem~\ref{thm:transversality} we have that $\chi(\KG^r(\overline K,L)) \ge D-d$. Indeed, if $\chi(\KG^r(\overline K,L)) \le D-1-d$, then
	Theorem~\ref{thm:transversality} would imply that $f$ identifies $r$ points from $r$ pairwise disjoint faces. 
	Now $D = \frac{n+t-1 - \alpha}{r-1}$ for some $0 \le \alpha \le r-2$, and $d = \frac{r(k-2)+t+1}{r-1}$. Thus
	$$\chi(\KG^r(\overline K,L)) \ge  \frac{n+t-1 - \alpha}{r-1} - \frac{r(k-2)+t+1}{r-1} = \frac{n-r(k-2)-(\alpha+2)}{r-1}
	\ge \frac{n-r(k-1)}{r-1}.$$
	
	The missing faces of $\overline K$ that are faces of~$L$ are precisely the $2$-stable $k$-sets in~$[n]$. Thus the
	hypergraph $\KG^r(\overline K, L)$ is~$\KG^r(k,n)_{2-\textrm{stab}}$.
\end{proof}

\begin{remark}
	For the $r=2$ case (that is, Kneser graphs) Meunier~\cite{meunier2011} showed that for $n \ge sk$ the $s$-stable 
	Kneser graph $\KG^2(k,n)_{s-\textrm{stab}}$ can be properly colored with $n -s(k-1)$ colors and conjectured that 
	this is in fact the chromatic number. This was confirmed by Jonsson~\cite{jonsson2012} for $s \ge 4$ and $n$ 
	sufficiently large, and by Chen~\cite{chen2015} for even $s$ and any~$n$.
\end{remark}

We can extend the reasoning above by using the full generality of Theorem~\ref{thm:hell}. If, in addition, we are somewhat more 
careful how to order the points $p_1, \dots, p_{n+t}$ we obtain a result for $3$-stable Kneser hypergraphs:

\begin{theorem}
\label{thm:almost-3-stable}
	Let $r\ge 4$, $r \notin \{6,12\}$, $k\ge 1$, and $n \ge rk$ be integers, and if $r = 4$ let either $n$ be even or $k \not\equiv 2 \mod 3$. 
	Then $$\chi(\KG^r(k,n)_{3-\textrm{stab}}) = \Big\lceil \frac{n-r(k-1)}{r-1} \Big\rceil.$$
\end{theorem} 

\begin{proof}
	We follow the same steps as in the proof of Theorem~\ref{thm:stable}. First let $r \ge 5$ be a prime power. 
	Recall that $K$ was defined as the 
	simplicial complex on vertex set~$[n]$ whose missing faces are precisely the almost $2$-stable $k$-subsets of~$[n]$. 
	Let $\widehat K$ be the simplicial complex on vertex set~$[n]$ obtained from $K$ by also adding any subset
	$\sigma \subset [n]$ of cardinality at least~$k$ and with $\{1,3\} \subset \sigma$, or $\{2,4\} \subset \sigma$, 
	or $\{n-2,n\} \subset \sigma$ to $K$ as a face. With $t$ defined as before (i.e., integer $d$ satisfies 
	$r(k - 1) - 1 \ge (r - 1)d > r(k - 2)$  and $t = (r-1)d-r(k-2)-1$) let $\overline K = \widehat K * \Delta_{t-1}$,
	where we identify the vertex set of $\Delta_{t-1}$ with $\{n+1, n+2, \dots, n+t\}$.
	
	Let $p_1, \dots, p_{n+t} \in \R^d$ be a sequence of points provided by Lemma~\ref{lem:colorful}. Call the first
	$\lfloor \frac{t}{2} \rfloor$ points $q_1, \dots, q_{\lfloor t/2 \rfloor}$, and the last $\lceil \frac{t}{2} \rceil$ points
	$q_{\lfloor t/2 \rfloor +1}, \dots, q_t$. This leaves $n$ points of $p_1, \dots, p_{n+t}$ in the middle that we will
	denote $p'_1, \dots, p'_n$. Now define the facewise linear map $f \colon \overline K \longrightarrow \R^d$ 
	by mapping vertex $i \in [n]$ to~$p'_i$ and vertex $i \in \{n+1, n+2, \dots, n+t\}$ to $q_{i-n}$. Let us again not
	distinguish between vertices of $\overline K$ and their image points in~$\R^d$.
	
	Let $\sigma_1, \dots, \sigma_r$ be pairwise disjoint faces of~$\overline K$ and suppose that $f(\sigma_1) \cap
	\dots \cap f(\sigma_r) \ne \emptyset$. Let us again assume w.l.o.g. that $|\bigcup_i \sigma_i| = (r-1)(d+1)+1$.
	As before one of the faces $\sigma_j$ satisfies that $\sigma_j \cap \{p'_1, \dots, p'_n\}$ has cardinality at least~$k$.
	Now since $\sigma_1 \sqcup \dots \sqcup \sigma_r$ is a colorful partition of those points in $p_1, \dots, p_{n+t}$ that occur 
	in one $\sigma_i$ the face $\sigma_j$
	cannot contain two successive elements of~$[n]$. It can also contain at most one element within the first $r$ points
	$p_1, \dots, p_r$ and at most one element within the last $r$ points $p_{n+t-r+1}, \dots, p_{n+t}$ by definition of
	colorful partition. Since $t \le r-2$ we have that $\{p'_1, \dots, p'_4\} \subset \{p_1, \dots, p_r\}$ -- that is, 
	$4+\lfloor \frac{t}{2} \rfloor \le r$ since $r \ge 5$ -- and $\{p'_{n-2}, p'_n\} \subset \{p_{n+t-r+1}, \dots, p_{n+t}\}$. 
	Thus, $\{1,3\} \nsubseteq \sigma_j$, $\{2,4\} \nsubseteq \sigma_j$, and $\{n-2,n\} \nsubseteq \sigma_j$, so $\sigma_j$
	is a missing face of $\overline K$ --- a contradiction.
	
	Let $G$ be the graph on vertex set $[n+t]$ with edges $\{3,5\}, \{4,6\}, \dots, \{n-3,n-1\}$ and $\{1,n\}, \{2,n\}$, $\{1, n-1\}$.
	In particular, $G$ is a disjoint union of paths. Let $L$ be the simplicial complex obtained from the simplex on vertex
	set $[n+t]$ by deleting any edge in~$G$. Now as in the proof of Theorem~\ref{thm:stable} and using Theorem~\ref{thm:hell}
	we get that any continuous map $F\colon L \longrightarrow \R^{D-1}$ identifies $r$ points from $r$ pairwise disjoint faces
	for $D = \lfloor \frac{n+t-1}{r-1} \rfloor$. Then as before 
	$$\chi(\KG^r(\overline K,L)) \ge \frac{n-r(k-1)}{r-1}.$$
	
	Now observe that $\KG^r(\overline K,L) = \KG^r(k,n)_{3-\textrm{stab}}$. The same proof works for $r = 4$ if $4+\lfloor \frac{t}{2} \rfloor \le 4$,
	or equivalently $t = 3d-4(k-2)-1 \le 1$, which is true unless $3d = 4(k-1)-1$, that is, unless $k \equiv 2 \mod 3$. If $r = 4$ and 
	$t = 2$, then for even $n$ the following modification of the proof still works: add only those $\sigma \subset [n]$ to $K$ that
	contain $\{1,3\}$ or $\{n-2,n\}$ to obtain~$\widehat K$. We then have to exclude faces containing $\{2,4\}$ via the graph~$G$,
	that is, we add the edge $\{2,4\}$ to~$G$. If $n$ is even $G$ is a path, so we can still apply Theorem~\ref{thm:hell}. 
	The case of general $r$ now follows by an induction using Lemma~\ref{lem:primes}.
\end{proof}

Our methods also yield an approximation to the following conjecture of Ziegler~\cite{ziegler2002} and Alon, Drewnowski, 
and \L uczak~\cite{alon2009}:

\begin{conjecture}
\label{conj:stable}
	For $n \ge rk$ we have that $$\chi(\KG^r(k,n)_{r-\textrm{stab}}) = \Big\lceil \frac{n-r(k-1)}{r-1} \Big\rceil.$$
\end{conjecture}

Schrijver's theorem~\cite{schrijver1978} is the $r=2$ case of this conjecture. It also holds for $r$ a power of two, since Alon, Drewnowski, 
and \L uczak~\cite{alon2009} found a combinatorial reduction that shows if Conjecture~\ref{conj:stable} holds for $r = p$ and $r = q$
then it also holds for the product $r = pq$. Denote by $\KG^r(k,n)_{\widetilde{2-\textrm{stab}}}$ the subhypergraph of $\KG^r(k,n)$ induced by
those vertices corresponding to almost $2$-stable $k$-sets. Recall that a set $\sigma \subset [n]$ is called \emph{almost $2$-stable}
if $|i-j|\ne 1$ for any $i,j \in \sigma$. The following theorem implies that for certain parameters we find optimal lower bounds for the
chromatic number of the subhypergraph $H \subset \KG^r(k,n)$ whose vertices are induced by almost $2$-stable $k$-sets that arise as
unions of two $r$-stable sets of cardinality~$k/2$.

\begin{theorem}
	Let $r \ge 2$ be a prime, and let $k \ge 1$, and $n \ge rk$ be integers. Let $C_1 \sqcup \dots \sqcup C_m$ be a partition of $[n]$
	with $|C_i| \le r-1$.	
	Let $H \subset \KG^r(k,n)_{\widetilde{2-\textrm{stab}}}$ be the subhypergraph that only contains those almost $2$-stable $k$-sets 
	$\sigma \subset [n]$ as vertices that satisfy $|\sigma \cap C_i| \le 1$ for all~$i$. Then $\chi(H) = \lceil \frac{n-r(k-1)}{r-1} \rceil$.
\end{theorem} 

\begin{proof}
	Repeat the proof of Theorem~\ref{thm:stable}, but now for the complex $L$ use $C_1* \dots * C_m * \Delta_{t-1}$, the join of
	the sets in the partition of~$[n]$ and a small simplex $\Delta_{t-1}$ with $t$ chosen as before. 
	The existence of the required $r$-fold intersection point amon pairwise disjoint faces is guaranteed by Theorem~\ref{thm:bmz}
	--- now only for $r$ a prime. The hypergraph $H$ is~$\KG^r(K,L)$.
\end{proof}

For $n$ divisible by $r-1$ let the $C_i$ be blocks of $r-1$ consecutive integers. If $k$ is even, say $k = 2\ell$, then any $k$-subset
$\sigma \subset [n]$ with $|\sigma \cap C_i| \le 1$ for all~$i$ can be split into two $r$-stable $\ell$-subsets: if $i_1 < i_2 < \dots < i_k$
are the elements of~$\sigma$, split it into $\{i_1, i_3, \dots, i_{k-1}\}$ and $\{i_2, i_4, \dots, i_k\}$. We explicitly formulate this as
a corollary:

\begin{corollary}
	Let $r \ge 2$ be a prime, $n$ divisible by $r-1$, $k = 2\ell$ for some integer $\ell \ge 1$, and $n \ge rk$. Let $H \subset \KG^r(k,n)$
	be the subhypergraph induced by vertices corresponding to almost $2$-stable $k$-sets that arise as the union of two $r$-stable
	$\ell$-sets whose elements alternate. Then $\chi(H) = \lceil \frac{n-r(k-1)}{r-1} \rceil$.
\end{corollary}

\section{Eliminating local effects on the colorabilty defect}
\label{sec:nonlocal}

Let $\F$ be a family of subsets of~$[n]$. The \emph{$r$-colorabilty defect} $\cd^r(\F)$ of $\F$ is defined as
$$\cd^r(\F) = n - \max\{\sum_{i=1}^r |A_i| \: : \: A_1, \dots, A_r \subset [n] \ \text{pairwise disjoint and} \ 
F \not\subset A_i \ \text{for all} \ F \in \F \ \text{and} \ i \in [r]\}.$$

For an arbitrary set system $\F$ we denote by $\KG^r(\F)$ the \emph{generalized Kneser hypergraph} whose 
vertices are the elements of~$\F$ and whose hyperedges are $r$-tuples of pairwise disjoint sets. This is only
slightly more general than the hypergraphs $\KG^r(K)$ for a simplicial complex~$K$, where the vertices are
determined by the missing faces of~$K$. For any system $\F$ of subsets of~$[n]$ such that for $F \ne F' \in \F$ 
neither $F \subset F'$ nor $F' \subset F$, there is a unique simplicial complex $K$ on vertex set $[n]$ whose
missing faces are precisely the sets in~$\F$. Moreover, for an arbitrary system of sets $\F$, if $\F' \subset \F$ denotes
the system of inclusion-minimal sets in~$\F$, then $\chi(\KG^r(\F)) = \chi(\KG^r(\F'))$. 
See~\cite[Lemma~4.1 and~4.2]{frick2017} for details. Hence, for the purposes of
chromatic numbers we may and will tacitly assume that $\F$ contains no two sets ordered by inclusion. Those
set systems are precisely the systems that arise as missing faces of simplicial complexes.

A usual progression has been that a result that was first established for Kneser hypergraphs later gets extended 
to generalized Kneser hypergraphs of arbitrary set systems: Lov\'asz established Kneser's conjecture, that 
$\chi(\KG^2(k,n)) = n-2(k-1) = \cd^2(\KG^2(k,n))$, which was extended to arbitrary set systems -- $\chi(\KG^2(\F)) \ge \cd^2(\KG^2(\F))$ --
by Dolnikov~\cite{dolnikov1988}. In this language the generalization of Kneser's conjecture to hypergraphs due to Alon, Frankl, and Lov\'asz 
-- Theorem~\ref{thm:afl} -- can be phrased as $\chi(\KG^r(k,n)) = \lceil \frac{1}{r-1}\cd^r(\KG^r(k,n)) \rceil$. The 
extension to set systems is due to K\v r\' i\v z~\cite{kriz1992}:

\begin{theorem}[K\v r\' i\v z~\cite{kriz1992}]
\label{thm:kriz}
	$$\chi(\KG^r(\F)) \ge \Big\lceil \frac{1}{r-1}\cd^r(\F) \Big\rceil$$
\end{theorem}

This raises the questions whether Schrijver's result also holds in this generalized setting. For $\F$ a system of subsets 
of~$[n]$, denote by $\F_{s-\textrm{stab}}$ the subsystem of $s$-stable sets in~$\F$. Then Schrijver's theorem guarantees that 
$\chi(\KG^2(\binom{[n]}{k}_{2-\textrm{stab}})) \ge \cd^2(\binom{[n]}{k})$. For $n = 5$ let $\F=\{\{1,2\},\{2,3\},\{3,4\},\{4,5\},\{5,1\}\}$.
Then $\cd^2(\F) = 1$, while $\chi(\KG^2(\F_{2-\textrm{stab}})) = 0$, so such an extension of Schrijver's result to arbitrary set systems 
cannot hold in general. However, this extension holds in full generality in the hypergraph setting, see Theorem~\ref{thm:hyp-schrijver},
and almost holds for $r=2$, see Theorem~\ref{thm:gen-schrijver}.

\begin{theorem}
\label{thm:hyp-schrijver}
	Let $\F$ be a system of subsets of~$[n]$ and $r \ge 3$. Then $$\chi(\KG^r(\F_{2-\textrm{stab}})) \ge \Big\lceil \frac{1}{r-1}\cd^r(\F) \Big\rceil.$$
\end{theorem}

\begin{proof}
	First assume that $r \ge 3$ is a prime power.
	Let $K$ be the simplicial complex on vertex set~$[n]$ with missing faces~$\F_{\widetilde{2-\textrm{stab}}}$, the collection of sets
	in~$\F$ that are almost $2$-stable. Denote by $K'$ the simplicial
	complex on vertex set $[n]$ with missing faces~$\F$.
	Let $$M = \max\{\sum_{i=1}^r |A_i| \: : \: A_1, \dots, A_r \subset [n] \ \text{pairwise disjoint and} \ 
	F \not\subset A_i \ \text{for all} \ F \in \F \ \text{and} \ i \in [r]\}.$$ 
	In particular, the sets $A_i$ in the definition of $M$ above determine faces of the complex~$K'$.
	Choose the integer $d$ as the least integer such that $M < (r-1)(d+1)+1$, i.e., $M+m = (r-1)(d+1)+1$ for some $0 < m \le r-1$.
	Define the simplicial complex $\overline K'$ to be the join $K' * \Delta_{m-2}$. Then any strong 
	general position map $f\colon \overline K' \longrightarrow \R^d$ satisfies $f(\sigma_1) \cap \dots \cap f(\sigma_r) = \emptyset$
	for any $r$ pairwise disjoint faces $\sigma_1, \dots, \sigma_r$ of~$\overline K'$. By definition of $M$ the faces $\sigma_i$
	may in total only involve at most $M$ vertices of~$K$. The complex $\overline K'$ has $m-1$ additional vertices, so
	the faces $\sigma_i$ involve less than $M+m = (r-1)(d+1)+1$ vertices, thus the sets $f(\sigma_1), \dots, f(\sigma_r)$ do not have
	a common point of intersection by a codimension count.
	
	Now let $\overline K$ be the join $K * \Delta_{m-2}$, and let $\widetilde f\colon \overline K \longrightarrow \R^d$ be a general position map
	that maps the vertices of $\overline K$ to pairwise distinct points on the stretched moment curve according to Lemma~\ref{lem:colorful}.
	Suppose there were $r$ pairwise disjoint faces $\sigma_1, \dots, \sigma_r$ of $\overline K$ such that 
	$\widetilde f(\sigma_1) \cap \dots \cap \widetilde f(\sigma_r) \ne \emptyset$. Then w.l.o.g. the $\sigma_i$ involve precisely 
	$(r-1)(d+1)+1$ vertices, and necessarily at least one $\sigma_j$ is not contained in~$\overline K'$. Such a face $\sigma_j$
	must contain two successive integers in $[n]$, in contradiction to Lemma~\ref{lem:colorful}.
	
	Let $L$ be the simplicial complex on the same vertex set as $\overline K$ that is obtained from the simplex by deleting 
	the edge $\{1,n\}$. We now proceed similar to the proof of Theorem~\ref{thm:stable}. Let $D = \lfloor \frac{n+m-2}{r-1} \rfloor$.
	Then by Theorem~\ref{thm:hell} any continuous map $F\colon L \longrightarrow \R^{D-1}$ identifies $r$ points from $r$ pairwise
	disjoint faces. Thus by Theorem~\ref{thm:transversality} we obtain the lower bound $\chi(\KG^r(\overline K, L)) \ge D-d$.
	It is left to show that $(r-1)(D-d) \ge \cd^r(\F) = n-M$. Writing $D = \frac{n+m-2-\alpha}{r-1}$ for some $0 \le \alpha \le r-2$ we
	need to show that $n+m-2-\alpha +M-(r-1)d \ge n$, which is equivalent to $M+m \ge (r-1)d+\alpha+2$. Since $M+m = (r-1)(d+1)+1$
	this follows from $\alpha \le r-2$.
	
	The case for general $r$ follows by induction on prime power factors very similar to the proof of K\v r\' i\v z~\cite{kriz2000c}. We have to be slightly more
	careful for those $r$ with precisely one two in their prime factorization. We essentially repeat K\v r\' i\v z' proof with a few changes. 
	Let $r = pq$ where we have shown the claimed lower bound for the chromatic number for $\KG^p(\F_{2-\textrm{stab}})$ and for
	$\KG^q(\F_{2-\textrm{stab}})$ for any set system $\F$ already. Let us moreover assume that $q\ne 2$. Write $t = \chi(\KG^r(\F_{2-\textrm{stab}}))$
	and assume for contradiction that $t-1 < \frac{1}{r-1}\cd^r(\F)$, that is, $(r-1)(t-1) < \cd^r(\F)$.
	
	For any subset $E \subset [n]$ denote by $\F|_E$ the set system $\{F\cap E \: : \: F\in \F\}$. 
	Let $\Gamma = \{E \subset [n] \: : \: \cd^q(\F|_E) > (q-1)(t-1)\}$. Then $\cd^p(\Gamma) > (p-1)(t-1)$.
	Otherwise we could find $A_1, \dots, A_p \subset [n]$ with no subset in~$\Gamma$ such that
	$|[n]\setminus\bigcup_i A_i| \le (p-1)(t-1)$. Since the $A_i$ are not in $\Gamma$ we know that $\cd^q(\F|_{A_i}) \le (q-1)(t-1)$,
	and so we can find $B_{i,1}, \dots, B_{i,q} \subset A_i$ with no subset in $\F$ and $|A_i \setminus \bigcup_j B_{i,j}| \le (q-1)(t-1)$.
	There are $pq = r$ sets $B_{i,j}$ in total and the complement of their union in $[n]$ has at most $p(q-1)(t-1)+(p-1)(t-1)$
	elements, which is equal to $(r-p)(t-1)+(p-1)(t-1) = (r-1)(t-1)$, which is a contradiction to $\cd^r(\F) > (r-1)(t-1)$.
	
	Given an arbitrary $(t-1)$-coloring of $\F$ we need to show that there are $r$ pairwise disjoint sets of $\F_{2-\textrm{stab}}$ 
	that receive the same color. By induction hypothesis we have that $\chi(\KG^q((\F|_E)_{2-\textrm{stab}})) > t-1$ for any $E \in \Gamma$.
	Thus we can find $q$ pairwise disjoint sets $X_{E,1}, \dots, X_{E,q} \in (\F|_E)_{2-\textrm{stab}}$ all colored by~$i_E$.
	Associating color $i_E$ to set $E$ defines a $(t-1)$-coloring of~$\Gamma$. Now since $\cd^p(\Gamma) > (p-1)(t-1)$
	we have by induction hypothesis (even for $p = 2$ by Theorem~\ref{thm:kriz}) that there are $p$ pairwise disjoint $E_1, \dots, E_p \in \Gamma$ 
	that receive the same color. Then $X_{E_1,1}, \dots, X_{E_1,q}, \dots, X_{E_p,1}, \dots, X_{E_p,q}$ are $r$ pairwise 
	disjoint sets in $\F_{2-\textrm{stab}}$ that all receive the same color, as desired.
\end{proof}

For $r = 2$ the proof above still works as long as we replace the simplicial complex $L$ with the simplex on the vertex
set of $\overline K$ and use the topological Radon theorem. We then obtain a lower bound for the chromatic number
of generalized Kneser graphs, but instead of $\F_{2-\textrm{stab}}$ this lower bound holds for the set system $\F_{\widetilde{2-\textrm{stab}}}$
that contains all almost $2$-stable sets in~$\F$.

\begin{theorem}
\label{thm:gen-schrijver}
	Let $\F$ be a system of subsets of~$[n]$. Then $\chi(\KG^2(\F_{\widetilde{2-\textrm{stab}}})) \ge \cd^2(\F)$.
\end{theorem}

\begin{remark}
	It was pointed out by Fr\'ed\'eric Meunier that a weaker version of Theorem~\ref{thm:hyp-schrijver} for almost stable 
	sets can also be deduced from recent work of Alishahi and Hajiabolhassan~\cite{alishahi2016} and their notion of
	\emph{alternation number} of a Kneser hypergraph.
\end{remark}

Let $\F = \binom{[n]}{k}$ be the system of all $k$-subsets of~$[n]$. Then $\cd^r(\F) = n-r(k-1)$, and thus
Theorem~\ref{thm:kriz} is an extension of Theorem~\ref{thm:afl} that $\chi(\KG^r(k,n)) \ge \lceil \frac{n-r(k-1)}{r-1} \rceil$
to arbitrary set systems. However, these general lower bounds are significantly worse for stable Kneser hypergraphs:
Any set of $2k-2$ consecutive integers in $[n]$ cannot contain a $2$-stable $k$-set, and thus for $\F$ the collection
of $2$-stable $k$-sets in $[n]$ the colorability defect is $\cd^r(\F) \le n-2r(k-1)$. Tight lower bounds for the chromatic 
numbers of stable Kneser hypergraphs cannot be obtained from the general result of Theorem~\ref{thm:kriz}. 

Here we extend Theorem~\ref{thm:kriz} for $r$ a power of a prime to show that these ``local effects'' have no influence on the 
chromatic number of a general Kneser hypergraph. We do this in two ways: first, we define a topological analogue $\tcd^r(\F)$ of the
colorability defect~$\cd^r(\F)$, which satisfies $\tcd^r(\F) \ge \cd^r(\F)$ and still gives a lower bound for the chromatic number
of~$\chi(\KG^r(\F))$ for $r$ a prime power. The quantity $\tcd^r(\F)$ takes the global topology of the set system $\F$ into 
account, while $\cd^r(\F)$ is defined purely combinatorially. The second way of eliminating local effects of a set system~$\F$ 
that result in loose lower bounds for $\chi(\KG^r(\F))$ will be to extend the results above to the $s$-stable setting.

The \emph{topological $r$-colorability defect} $\tcd^r(\F)$ of a system $\F$ of subsets of~$[n]$ is defined as the maximum of 
$N - (r-1)(d+1)$, where $N$ is the number of vertices of a simplicial complex $K$ whose missing faces are precisely the
inclusion-minimal sets in~$\F$ and $d$ is the smallest dimension such that there is a continuous map $f\colon K \longrightarrow \R^d$ that
does not identify $r$ points from $r$ pairwise disjoint faces. A general position map $f\colon K \longrightarrow \R^d$
shows that $\cd^r(\F) \le \tcd^r(\F)$:

\begin{lemma}
	Let $\F$ be a system of subsets of~$[n]$. Then $\cd^r(\F) \le \tcd^r(\F)$.
\end{lemma}

\begin{proof}
	Let $K$ be the simplicial complex on vertex set $[n]$ with missing faces~$\F$. 
	Let $$M = \max\{\sum_{i=1}^r |A_i| \: : \: A_1, \dots, A_r \subset [n] \ \text{pairwise disjoint and} \ 
	F \not\subset A_i \ \text{for all} \ F \in \F \ \text{and} \ i \in [r]\}.$$ 
	In particular, the sets $A_i$ in the definition of $M$ above determine faces of the complex~$K$.
	Choose the integer $d$ as the least integer such that $M < (r-1)(d+1)+1$, i.e., $M+m = (r-1)(d+1)+1$ for some $0 < m \le r-1$.
	Define the simplicial complex $\overline K$ to be the join $K * \Delta_{m-2}$. Then any strong 
	general position map $f\colon \overline K \longrightarrow \R^d$ satisfies $f(\sigma_1) \cap \dots \cap f(\sigma_r) = \emptyset$
	for any $r$ pairwise disjoint faces $\sigma_1, \dots, \sigma_r$ of~$\overline K$. By definition of $M$ the faces $\sigma_i$
	may in total only involve at most $M$ vertices of~$K$. The complex $\overline K$ has $m-1$ additional vertices, so
	the faces $\sigma_i$ involve less than $M+m = (r-1)(d+1)+1$, thus the sets $f(\sigma_1), \dots, f(\sigma_r)$ do not have
	a common point of intersection by a codimension count.
	By definition $\tcd^r(\F) \ge n+m-1-(r-1)(d+1)$ and this is at least $n-M = \cd^r(\F)$ since $M +m = (r-1)(d+1)+1$.
\end{proof}

\begin{theorem} 
\label{thm:top}
	For $r$ a power of a prime
	$$\chi(\KG^r(\F)) \ge \Big\lceil \frac{1}{r-1}\tcd^r(\F) \Big\rceil.$$
\end{theorem}

\begin{proof}
	By definition of topological $r$-colorabilty defect there is a map $f\colon K \longrightarrow \R^d$ that does not
	identify $r$ points from $r$ pairwise disjoint faces, where $K$ is a simplicial complex on $N$ vertices whose
	missing faces are the sets in~$\F$ and $\tcd^r(\F) = N - (r-1)(d+1)$. For any continuous map
	$F\colon \Delta_{N-1} \longrightarrow \R^{D-1}$ there are $r$ pairwise disjoint faces of $\Delta_{N-1}$ whose
	images have a common point of intersection, where $D = \lfloor \frac{N-1}{r-1} \rfloor$.
	
	By Theorem~\ref{thm:transversality} we have that $\chi(\KG^r(K)) \ge D-d$. It remains to show that $(r-1)(D-d) \ge \tcd^r(\F) = N - (r-1)(d+1)$.
	This is equivalent to $(r-1)\lfloor \frac{N-1}{r-1} \rfloor \ge N - (r-1)$, which is evidently true.
\end{proof}

We make two remarks that easily follow from~\cite[Examples~4.9 and~4.11]{frick2017}. 
Theorem~\ref{thm:top} fails for any $r$ that is not a prime power. Let $\F$ be $\binom{[n]}{k}$, the system of $k$-subsets of~$[n]$,
where $n = rk-1$ for some integer $r\ge 6$ with at least two distinct prime divisors. The hypergraph $\KG^r(\F)$ does not
contain any hyperedges and thus has chromatic number $\chi(\KG^r(\F)) = 1$. Let $K$ be the complex on vertex set $[n]$
with missing faces~$\F$, that is, $K = \Delta_{n-1}^{(k-2)}$. Suppose that $k-2 = (r-1)\ell$ for some integer $\ell \ge 3$. 
Then there is a map $f\colon K \longrightarrow \R^{r\ell}$ that does not identify $r$ points from $r$ pairwise disjoint faces~\cite{mabillard2015}.
One now easily checks that $\tcd^r(\F) \ge n-(r-1)(r\ell+1) = r$, and thus $\lfloor \frac{1}{r-1}\tcd^r(\F) \rfloor \ge 2$.

The gap $\tcd^r(\F) - \cd^r(\F)$ can be arbitrarily large. For example, let $K$ be a triangulation of a $d$-ball on $n$ vertices
and denote by $\F$ the system of its missing faces. Since $K$ embeds into $\R^d$ we have that $\tcd^2(\F) \ge n -(d+1)$,
while $\cd^2(\F) = n- 2(d+1)$, since the facets of $K$ have no subset in~$\F$.

Let $\F$ be a system of subsets of $[n]$, and let $C_1 \sqcup \dots \sqcup C_m$ be a partition of $[n]$ into $m$ sets. 
We call a set $F \in \F$ \emph{transversal} if $|F\cap C_i| \le 1$ for every~$i$. Denote by $\F_{\textrm{transversal}}$
the system of all transversal sets in~$\F$ with respect to the partition $C_1 \sqcup \dots \sqcup C_m$.

\begin{theorem} 
\label{thm:top-opt}
	Let $r\ge 2$ be a prime, and let $\F$ be a system of subsets of~$[n]$, and $C_1 \sqcup \dots \sqcup C_m$ 
	a partition of~$[n]$ into sets of size at most~$r-1$. Then
	$$\chi(\KG^r(\F_{\textrm{transversal}})) \ge \Big\lceil \frac{1}{r-1}\tcd^r(\F) \Big\rceil.$$
\end{theorem}

\begin{proof}
	Repeat the proof of Theorem~\ref{thm:top}, but now apply Theorem~\ref{thm:bmz} to the simplicial complex
	obtained from the simplex on the vertex set of~$K$ by deleting all edges with both endpoints in one~$C_i$.
\end{proof}

Theorem~\ref{thm:top-opt} is wrong for any $r$ with at least two distinct prime divisors since the weaker Theorem~\ref{thm:top}
is wrong for those~$r$. We conjecture that the primality of $r$ is required. 
Theorem~\ref{thm:top-opt} is a common generalization and extension of the results of K\v r\' i\v z (for $r$ a prime) and Schrijver as well as the
hypergraph version of Schrijver's result, Theorem~\ref{thm:hyp-schrijver}. The proof of Theorem~\ref{thm:hyp-schrijver} shows
that it is implied by Theorem~\ref{thm:top-opt} (for $r \ge 3$ a prime), where we bound $\tcd^r$ from below by exhibiting an affine 
map that sends vertices to points along the moment curve, and sets $C_i$ of the partition have size one, except for one set that
contains $1$ and~$n$. For the $r=2$ setting we are not allowed to have sets $C_i$ of size two, and we can only 
restrict to almost $2$-stable sets. Thus, Theorem~\ref{thm:gen-schrijver} is also an easy corollary of Theorem~\ref{thm:top-opt}.

We will now prove tight lower bounds for chromatic numbers of $s$-stable Kneser hypergraphs.
For a graph $G$ and a vertex $v$ of~$G$ denote by $N(v)$ the \emph{neighborhood of~$v$}, that is, all vertices at distance one
from~$v$. Similarly denote by $N^2(v)$ the set of vertices at distance precisely two from~$v$. Engstr\"om proved the following
Tverberg-type result for sparse subcomplexes of a simplex determined by a graph $G$ of forbidden edges.

\begin{theorem}[Engstr\"om~\cite{engstrom2011}]
\label{thm:engstrom}
	Let $r \ge 2$ be a prime power, $d \ge 1$ an integer, and $N =(r-1)(d+1)$. Let $G$ be a graph on the vertex set of~$\Delta_N$
	that satisfies $2|N(v)|+|N^2(v)| < r$ for every vertex~$v$. Then for any continuous map $f \colon \Delta_N \longrightarrow \R^d$ there are $r$
	pairwise disjoint faces $\sigma_1, \dots, \sigma_r$ of~$\Delta_N$ such that $f(\sigma_1) \cap \dots \cap f(\sigma_r) \ne \emptyset$
	and no $\sigma_i$ contains both endpoints of an edge of~$G$.
\end{theorem}

Using this result we obtain a strengthening of Theorem~\ref{thm:top}:

\begin{theorem}
\label{thm:top-stable}
	Let $s \ge 2$ be an integer and $r > 6s-6$ a prime power. Then for any system $\F$ of subsets of $[n]$ we have that
	$$\chi(\KG^r(\F_{s-\textrm{stab}})) \ge \Big\lceil \frac{\tcd^r(\F)}{r-1} \Big\rceil.$$
\end{theorem}

\begin{proof}
	The proof is the same as for Theorem~\ref{thm:top} with the strengthening that instead of concluding that
	$F\colon\Delta_{N-1} \longrightarrow \R^{D-1}$ identifies $r$ points from $r$ pairwise disjoint faces, we now
	apply the proper extension of Theorem~\ref{thm:engstrom}:	
	Let $G$ be the graph on vertex set~$[n]$ with an edge $(i,j)$ precisely if $i$ and $j$ are at distance at most $s-1$ in the
	cyclic order on~$[n]$. Notice that for every vertex $v$ of $G$ the neighborhood $N(v)$ has cardinality $2s-2$
	and $N^2(v)$ has cardinality $2s-2$ as well.
	Let $L$ be the simplicial complex obtained from the simplex on vertex set $[N]$ by deleting the 
	edges of~$G$. The missing faces of $K$ that are contained in~$L$ are precisely the $s$-stable sets in~$\F$. 	
	By Theorem~\ref{thm:engstrom} any continuous map $F\colon L \longrightarrow \R^{D-1}$ identifies $r$ points from 
	$r$ pairwise disjoint faces. 
\end{proof}

Since $\tcd^r(\binom{[n]}{k}) = n-r(k-1)$ -- this is the content of the proof of Theorem~\ref{thm:afl} -- we obtain as an immediate consequence:

\begin{corollary}
\label{cor:s-stable}
	Let $s \ge 2$ be an integer and $r > 6s-6$ a prime power. Then for integers $k \ge 1$ and $n \ge rk$ we have that
	$$\chi(\KG^r(k,n)_{s-\textrm{stab}}) = \Big\lceil \frac{n-r(k-1)}{r-1} \Big\rceil.$$
\end{corollary}

Certainly our methods can be pushed further with some additional effort. The special structure of the constraint graph for
the Tverberg-type result needed to prove Corollary~\ref{cor:s-stable} could enable us in principle to get the same result
for $r \sim 2s$ by topological connectivity alone.
Conjecture~\ref{conj:stable} is widely believed to be true. Here we conjecture its generalization to arbitrary set systems:

\begin{conjecture}
	Let $r \ge 3$ and let $\F$ be a set system. Then $$\chi(\KG^r(\F_{r-\textrm{stab}})) \ge  \Big\lceil \frac{\cd^r(\F)}{r-1} \Big\rceil.$$
\end{conjecture}

\section{Lower bounds for chromatic numbers via equivariant topology}
\label{sec:box}

The methods presented in the preceding sections have the shortcoming that they only apply to hypergraphs that are represented as Kneser 
hypergraphs, with vertices in correspondence with missing faces of a simplicial complex and hyperedges that contain $r$-tuples
of pairwise disjoint missing faces. Here we will remedy this shortcoming by lifting our construction to a box complex construction
for hypergraphs. Our results turn out to again improve upon the general lower bounds obtained by K\v r\' i\v z. This section also aims to
clarify how the methods of the preceding sections relate to the more classical box complex constructions.

Let $H$ be an $r$-uniform hypergraph. Let $\overline{H}$ be the smallest simplicial complex containing~$H$, that is, for any
hyperedge $\sigma \in H$ the complex $\overline{H}$ contains all subsets of~$\sigma$.
Denote by 
\begin{align*}
B(H) = \{A_1 \times \{1\} \cup \dots \cup A_r \times \{r\}
\: | \: A_i \subseteq V(H), A_i \cap A_j = \emptyset \ \text{for} \ i \ne j, \\
\sigma \in \overline H \ \text{for all} \ \sigma \subset V(H) \ \text{with} \ |\sigma\cap A_i| \le 1\}
\end{align*}
the \emph{box complex of~$H$}. In particular, the vertex set of $B(H)$ consists of $r$ copies of the vertex set of~$H$, provided that $H$
does not have isolated vertices.

Topological box complex approaches (or the more general Hom-complexes) to finding lower bounds for chromatic numbers of graphs are standard; see 
Matou\v sek and Ziegler~\cite{matousek2002} for a survey. Extensions to hypergraphs are not new either; see 
Lange and Ziegler~\cite{lange2007}, Iriye and Kishimoto~\cite{iriye2013}, and Alishahi~\cite{alishahi2016}. Our contribution is a linearization of the theory that also
works for $r$ that are not powers of primes; see Theorem~\ref{thm:equivariant}. Moreover, Theorem~\ref{thm:maps} explains the 
relation of our methods to the usual box complex approach.

\begin{lemma}[Volovikov~\cite{volovikov1996-2}]
\label{lem:volovikov}
	Let $p$ be a prime and $G = (\Z/p)^n$ an elementary abelian $p$-group. Suppose that $X$ and $Y$ are fixed-point free 
	$G$-spaces such that $\widetilde{H}^i(X;\Z/p) \cong 0$ for all $i \le n$ and $Y$ is an $n$-dimensional cohomology sphere 
	over~$\Z/p$. Then there does not exist a $G$-equivariant map $X \longrightarrow Y$.
\end{lemma}

We denote by the standard representation of the symmetric group $S_r$ by~$W_r$, that is, $$W_r = \{(x_1, \dots, x_r) \in \R^r
\: : \: \sum x_i = 0\}$$ with the action by $S_r$ that permutes coordinates. We denote by $K^{*r}$ the \emph{$r$-fold join} of a
simplicial complex~$K$, whose faces are joins $\sigma_1 * \dots * \sigma_r$ of $r$ faces of~$K$. If the faces $\sigma_i$ are
required to be $s$-wise disjoint, that is, any $s$ of them have no vertex in common, then the resulting complex is denoted
$K^{*r}_{\Delta(s)}$, the \emph{$s$-wise deleted $r$-fold join}. We write $K^{*r}_\Delta$ for $K^{*r}_{\Delta(2)}$. See 
Matou\v sek~\cite{matousek2008} for details and notation.

\begin{theorem}
\label{thm:equivariant}
	If $H$ is $c$-colorable then there exists an affine $S_r$-equivariant map $B(H) \longrightarrow W_r^{\oplus c} \setminus \{0\}$.
	In particular, if $r$ is a prime power and $B(H)$ is $[(r-1)(c-1)-1]$-connected, then $\chi(H) \ge c$.
\end{theorem}

\begin{proof}
	Let $f\colon V(H) \longrightarrow \{1, \dots, c\}$ be a proper coloring of~$H$. Think of color $i \in \{1, \dots, c\}$
	as the standard unit vector $e_i \in \R^c$. Denote by $\Delta$ the simplex on vertex set~$V(H)$. By affinely extending
	onto the faces of $\Delta$ we can think of $f$ as a simplicial map $f\colon \Delta \longrightarrow \Delta_{c-1}$.
	The box complex $B(H)$ is a subcomplex of~$\Delta^{*r}_\Delta$. Thus the $r$-fold join of~$f$ yields an
	$S_r$-equivariant map $F\colon B(H) \longrightarrow (\Delta_{c-1})^{*r}$. Since $f$ is a proper coloring of~$H$,
	the image of~$F$ is contained in~$(\Delta_{c-1})^{*r}_{\Delta(r)}$: otherwise there is a vertex $i$ of~$\Delta_{c-1}$
	(equivalently, a color in~$\{1, \dots, c\}$) and pairwise distinct vertices $v_1, \dots, v_r \in V(H)$ that are all colored
	by~$i$ and such that $\{v_1, \dots, v_r\}$ is a hyperedge of~$H$. But this is a contradiction to~$f$ being a proper
	coloring.
	
	Now, $(\Delta_{c-1})^{*r}_{\Delta(r)} \cong (\pt^{*c})^{*r}_{\Delta(r)} \cong (\pt^{*r}_{\Delta(r)})^{*c} \cong (\partial\Delta_{r-1})^{*c}
	\approx S^{(r-1)c-1}$. This join of boundaries of simplices is a polytope which can be realized with its full group of symmetries
	in $W_r^{\oplus c} \cong \R^{(r-1)c}$ such that the origin is in its interior. Hypergraph~$H$ cannot be $(c-1)$-colorable
	if there is no $S_r$-map $B(H) \longrightarrow W_r^{\oplus (c-1)} \setminus \{0\}$. In particular, such a map does not exist
	if $r$ is a prime power and $B(H)$ is $[(r-1)(c-1)-1]$-connected by Lemma~\ref{lem:volovikov}. 
\end{proof}

\begin{example}
	Let $r \ge 1$ and $c \ge 1$ be integers.
	Let $H$ be the hypergraph on vertex set~$[n]$ for $n = (r-1)(c-1)+1$ that has all $r$-element sets as hyperedges. 
	Then the chromatic number of $H$ is $\chi(H) = c$. To see that $c$ colors suffice, color the vertex set~$[n]$
	by splitting it into $c-1$ parts of size $r-1$ and color the last vertex by an additional color. The box complex $B(H)$
	is isomorphic to~${(\Delta_{n-1})^{*r}_\Delta}$. An $S_r$-equivariant map $B_r(H) \longrightarrow W_r^{\oplus c} \setminus \{0\}$
	exists if and only if $r$ has two distinct prime divisors; see \"Ozaydin~\cite{ozaydin1987}. Thus for prime powers we get the lower bound $\chi(H) \ge c$. 
	This topological approach does not yield an immediate tight lower bound for $r$ not a power of a prime. 
	However, there is no \emph{affine} $S_r$-equivariant map $B_r(H) \longrightarrow W_r^{\oplus c} \setminus \{0\}$ 
	for any~$r$. This follows from B\'ar\'any's colorful Carath\'eodory's theorem~\cite{barany1982}; see Sarkaria~\cite{sarkaria2000}. 
	Thus the linearized version in Theorem~\ref{thm:equivariant} is properly stronger than the usual topological approach.
\end{example}

The following theorem explains the relation of the approach of Section~\ref{sec:nonlocal} via geometric transversality results
to the box complex approach of Theorem~\ref{thm:equivariant}. The two approaches can be related using ideas of Sarkaria;
see in particular~\cite[Proof of Theorem~1.3]{sarkaria1991}.

\begin{theorem}
\label{thm:maps}
	Let $K \subset L$ be simplicial complexes and let $r$ be a prime. Further assume that there exists an $\Z/r$-equivariant map
	$f\colon K^{*r}_\Delta \longrightarrow S^N$, where the $N$-sphere $S^N$ has a free linear $\Z/r$-action. 
	If $\chi(\KG^r(K,L)) \le c$ then there exists an $\Z/r$-equivariant map $F\colon L^{*r}_\Delta \longrightarrow 
	S^{N+(r-1)c}$, where the $\Z/r$-action on $S^{N+(r-1)c}$ is free.
\end{theorem}

\begin{proof}
	Let $H$ be the hypergraph with vertices the nonfaces of $K$ contained in~$L$ (recall that the vertices of $\KG^r(K,L)$
	are the missing faces of $K$ in~$L$, that is, the minimal nonfaces) and hyperedges corresponding to $r$ pairwise disjoint
	nonfaces. Then $\KG^r(K,L) \subset H$ and $\chi(\KG^r(K,L)) = \chi(H)$: any vertex of $H$ is a nonface $\sigma$ of $K$ 
	that contains some missing face $\tau$ of~$K$; color $\sigma$ by the same color as~$\tau$.
	
	Denote by $\Sigma'$ the barycentric subdivision of the simplicial complex~$\Sigma$. 
	A vertex of~$(L^{*r}_\Delta)'$ corresponds to a join $\tau_1 * \dots * \tau_r$ of pairwise disjoint 
	faces in~$L$. Think of a join of faces in $L$ as the (disjoint) union~$\bigsqcup_i \tau_i \times \{i\}$. Denote by 
	$\tau_K$ the face $\bigsqcup_{i : \tau_i \in K} \tau_i \times \{i\}$ of $K^{*r}_\Delta$ and by $\tau_L$ the 
	face~${\bigsqcup_{i : \tau_i \notin K} \tau_i \times \{i\}}$. Each vertex of $B(H)$ corresponds to $\tau \times \{i\}$
	where $\tau \in L \setminus K$ and~${i \in [r]}$. Since the $\tau_i$ are pairwise disjoint, the face $\tau_L$
	is naturally a face of~$B(H)$, and thus a vertex of~$B(H)'$.
	By Theorem~\ref{thm:equivariant} there is an (affine) $S_r$-equivariant map $h\colon B(H) \longrightarrow W_r^{\oplus c} \setminus \{0\}$.
	
	Think of $S^N$ as equivariantly embedded into~$\R^{N+1}$.
	The map $F$ can now simply be defined as an affine map on $(L^{*r}_\Delta)'$ by sending 
	the vertex corresponding to the face $\tau_1 * \dots * \tau_r \in L^{*r}_\Delta$ to 
	$$(h(\tau_L), f(\tau_K)) \in W_r^{\oplus c} \oplus \R^{N+1},$$
	where we consider $\tau_L$ as a vertex of $B(H)'$ and $\tau_K$ as a vertex of~$(K^{*r}_\Delta)'$, provided that $\tau_L$
	and $\tau_K$ are nonempty. We define $h(\emptyset) = 0 = f(\emptyset)$.
	The map $F$ is $\Z/r$-equivariant and misses the origin of~${W_r^{\oplus c} \oplus \R^{N+1}}$. 
	To see the latter notice that it is not possible that both $\tau_K$ and $\tau_L$ are empty, and that a face of $(L^{*r}_\Delta)'$
	corresponds to a chain of faces $\tau^{(1)} \subset \tau^{(2)} \subset \dots \subset \tau^{(m)}$ of~$L^{*r}_\Delta$;
	the corresponding nonempty $\tau_K^{(i)}$ form a face of $(K^{*r}_\Delta)'$ (and thus $F$ does not hit $0$ in $\R^{N+1}$ 
	as long as $\tau_K^{(i)} \ne \emptyset)$, and similarly the corresponding nonempty $\tau_L^{(i)}$ form a face of~$B(H)'$
	(and thus $F$ does not hit $0$ in $W_r^{\oplus c}$ as long as $\tau_L^{(i)} \ne \emptyset)$. The map $F$ as defined 
	above thus retracts equivariantly to a map to the unit sphere~$S^{N+(r-1)c}$ in~${W_r^{\oplus c} \oplus \R^{N+1}}$.
\end{proof}

Since any map $K\longrightarrow \R^d$ that does not identify $r$ points from $r$ pairwise disjoint faces induces a
$\Z/r$-equivariant map $K^{*r}_\Delta \longrightarrow S(W_r^{\oplus(d+1)})$ as a retraction to the unit sphere of its
$r$-fold join, the results of the previous sections are special cases of Theorem~\ref{thm:maps} for $r$ a prime.
The theorem can be extended to prime powers by using fixed-point free actions of elementary abelian groups.
The methods of~\cite{frick2017} and the manuscript at hand can be seen as augmenting the box complex construction
by first constructing an equivariant map from a sufficiently large symmetric subcomplex. We avoid having to deal with 
obstruction theory by inducing these equivariant maps from continuous maps $K \longrightarrow \R^d$ that do not
identify points from $r$ pairwise disjoint faces.

\section*{Acknowledgements}

I would like to thank Fr\'ed\'eric Meunier for many helpful discussions about chromatic numbers of Kneser hypergraphs
and G\"unter M. Ziegler for good comments on a draft of this manuscript.
This material is based upon work supported by the National Science Foundation under Grant No. DMS--1440140
while the author was in residence at the Mathematical Sciences Research Institute in Berkeley, California, during the
Fall 2017 semester.

\bibliographystyle{amsplain}

\begin{thebibliography}{10}

\bibitem{alishahi2016}
Meysam Alishahi, \emph{Colorful subhypergraphs in uniform hypergraphs}, arXiv
  preprint arXiv:1605.06701 (2016).

\bibitem{alishahi2015}
Meysam Alishahi and Hossein Hajiabolhassan, \emph{{On the chromatic number of
  general Kneser hypergraphs}}, J. Combin. Theory, Ser.~B \textbf{115} (2015),
  186--209.

\bibitem{alon2009}
Noga Alon, Lech Drewnowski, and Tomasz {\L}uczak, \emph{{Stable Kneser
  Hypergraphs and Ideals in $\mathbb{N}$ with the Nikodym Property}}, Proc.
  Amer. Math. Soc. (2009), 467--471.

\bibitem{alon1986}
Noga Alon, Peter Frankl, and L{\'a}szl{\'o} Lov{\'a}sz, \emph{{The chromatic
  number of Kneser hypergraphs}}, Trans. Amer. Math. Soc. \textbf{298} (1986),
  no.~1, 359--370.

\bibitem{bajmoczy1979}
E.~G. Bajm{\'o}czy and Imre B{\'a}r{\'a}ny, \emph{{On a common generalization
  of Borsuk's and Radon's theorem}}, Acta Math. Hungar. \textbf{34} (1979),
  no.~3, 347--350.

\bibitem{barany1982}
Imre B{\'a}r{\'a}ny, \emph{{A generalization of Carath{\'e}odory's theorem}},
  Discrete Math. \textbf{40} (1982), no.~2-3, 141--152.

\bibitem{barany1981}
Imre B{\'a}r{\'a}ny, Senya~B. Shlosman, and Andr{\'a}s Sz{\H{u}}cs, \emph{On a
  topological generalization of a theorem of {T}verberg}, J. Lond. Math. Soc.
  \textbf{23} (1981), 158--164.

\bibitem{blagojevic2015}
Pavle V.~M. Blagojevi{\'c}, Florian Frick, and G{\"u}nter~M. Ziegler,
  \emph{{Barycenters of Polytope Skeleta and Counterexamples to the Topological
  Tverberg Conjecture, via Constraints}}, J. Europ. Math. Soc., to appear
  (2017).

\bibitem{blagojevic2009}
Pavle V.~M. Blagojevi{\'c}, Benjamin Matschke, and G{\"u}nter~M. Ziegler,
  \emph{{Optimal bounds for the colored Tverberg problem}}, J. Europ. Math.
  Soc. \textbf{17} (2015), no.~4, 739--754.

\bibitem{bukh2016}
Boris Bukh, Po-Shen Loh, and Gabriel Nivasch, \emph{{Classifying unavoidable
  Tverberg partitions}}, arXiv preprint arXiv:1611.01078 (2016).

\bibitem{chen2015}
Peng-An Chen, \emph{{On the Multichromatic Number of $s$-Stable Kneser
  Graphs}}, J. Graph Theory \textbf{79} (2015), no.~3, 233--248.

\bibitem{dolnikov1988}
Vladimir Dol'nikov, \emph{A certain combinatorial inequality}, Siberian
  Mathematical Journal \textbf{29} (1988), no.~3, 375--379.

\bibitem{engstrom2011}
Alexander Engstr{\"o}m, \emph{{A local criterion for Tverberg graphs}},
  Combinatorica \textbf{31} (2011), no.~3, 321--332.

\bibitem{erdos1973}
Paul Erd{\H o}s, \emph{Problems and results in combinatorial analysis}, Colloq.
  Internat. Theor. Combin. Rome, 1973, pp.~3--17.

\bibitem{frick2015}
Florian Frick, \emph{{Counterexamples to the topological Tverberg conjecture}},
  Oberwolfach Reports \textbf{12} (2015), no.~1, 318--321.

\bibitem{frick2017}
Florian Frick, \emph{Intersection patterns of finite sets and of convex sets}, Proc.
  Amer. Math. Soc. \textbf{145} (2017), no.~7, 2827--2842.

\bibitem{hell2008}
Stephan Hell, \emph{Tverberg's theorem with constraints}, J. Combin. Theory,
  Ser. A \textbf{115} (2008), no.~8, 1402--1416.

\bibitem{iriye2013}
Kouyemon Iriye and Daisuke Kishimoto, \emph{Hom complexes and hypergraph
  colorings}, Topology Appl. \textbf{160} (2013), no.~12, 1333--1344.

\bibitem{jonsson2012}
Jakob Jonsson, \emph{On the chromatic number of generalized stable {K}neser
  graphs}, 2012, manuscript, available at
  https://people.kth.se/~jakobj/doc/submitted/stablekneser.pdf.

\bibitem{kneser1955}
Martin Kneser, \emph{Aufgabe 360}, Jahresber. Dtsch. Math.-Ver. \textbf{2} (1955), 27.

\bibitem{kriz1992}
Igor K{\v r}{\' i}{\v z}, \emph{Equivariant cohomology and lower bounds for
  chromatic numbers}, Trans. Amer. Math. Soc. \textbf{33} (1992), 567--577.

\bibitem{kriz2000c}
Igor K{\v r}{\' i}{\v z}, \emph{{A correction to ``Equivariant cohomology and lower bounds for
  chromatic numbers''}}, Trans. Amer. Math. Soc. \textbf{352} (2000), no.~4,
  1951--1952.

\bibitem{lange2007}
Carsten Lange and G{\"u}nter~M. Ziegler, \emph{{On generalized Kneser
  hypergraph colorings}}, J. Combin. Theory, Ser.~A \textbf{114} (2007), no.~1,
  159--166.

\bibitem{lovasz1978}
L{\'a}szl{\'o} Lov{\'a}sz, \emph{{Kneser's conjecture, chromatic number, and
  homotopy}}, J. Combin. Theory, Ser. A \textbf{25} (1978), no.~3, 319--324.

\bibitem{mabillard2015}
Isaac Mabillard and Uli Wagner, \emph{{Eliminating Higher-Multiplicity
  Intersections, I. A Whitney Trick for Tverberg-Type Problems}}, arXiv
  preprint arXiv:1508.02349 (2015).

\bibitem{matousek2008}
Ji{\v{r}}{\'\i} Matou{\v{s}}ek, \emph{{Using the {Borsuk--Ulam} Theorem.
  {L}ectures on Topological Methods in Combinatorics and Geometry}}, second
  ed., Universitext, Springer-Verlag, Heidelberg, 2008.

\bibitem{matousek2002}
Ji{\v{r}}{\'\i} Matousek and G{\"u}nter~M. Ziegler, \emph{Topological lower bounds for the
  chromatic number: A hierarchy}, Jahresber. Dtsch. Math.-Ver. \textbf{106}
  (2004), 71--90.

\bibitem{meunier2011}
Fr{\'e}d{\'e}ric Meunier, \emph{{The chromatic number of almost stable Kneser
  hypergraphs}}, J. Combin. Theory, Ser.~A \textbf{118} (2011), no.~6,
  1820--1828.

\bibitem{ozaydin1987}
Murad {\"O}zaydin, \emph{Equivariant maps for the symmetric group}, Preprint,
  17 pages, \url{http://digital.library.wisc.edu/1793/63829}, 1987.

\bibitem{perles2014}
Micha~A. Perles and Moriah Sigron, \emph{Tverberg partitions of points on the moment curve}, Discrete
Comput. Geom. \textbf{57} (2017), no.~1, 56--70.

\bibitem{sarkaria1990}
Karanbir~S. Sarkaria, \emph{{A generalized Kneser conjecture}}, J. Combin.
  Theory, Ser. B \textbf{49} (1990), no.~2, 236--240.

\bibitem{sarkaria1991}
Karanbir~S. Sarkaria, \emph{A generalized {van Kampen--Flores} theorem}, Proc.
  Amer. Math. Soc. \textbf{11} (1991), 559--565.

\bibitem{sarkaria2000}
Karanbir~S. Sarkaria, \emph{{Tverberg partitions and Borsuk--Ulam theorems}},
  Pacific J. Math. \textbf{196} (2000), no.~1, 231--241.

\bibitem{schrijver1978}
Alexander Schrijver, \emph{{Vertex-critical subgraphs of Kneser-graphs}}, Nieuw
  Archief voor Wiskunde \textbf{26} (1978), 454--461.

\bibitem{volovikov1996-2}
Aleksei~Yu. Volovikov, \emph{{On a topological generalization of the Tverberg
  theorem}}, Math. Notes \textbf{59} (1996), no.~3, 324--326.

\bibitem{vucic1993}
Aleksandar Vu{\v{c}}i{\'c} and Rade~T. {\v{Z}}ivaljevi{\'c}, \emph{Note on a
  conjecture of {S}ierksma}, Discrete Comput. Geom. \textbf{9} (1993), no.~1,
  339--349.

\bibitem{ziegler2002}
G{\"u}nter~M. Ziegler, \emph{{Generalized Kneser coloring theorems with
  combinatorial proofs}}, Invent. Math. \textbf{147} (2002), no.~3, 671--691.

\end{thebibliography}


\end{document}